\documentclass[twoside,12pt]{amsart}
\usepackage[latin1]{inputenc}
\usepackage{amsmath, amsthm, amssymb, amsfonts, amscd}
\usepackage[english]{babel}
\usepackage[all]{xy}
\usepackage{url}
\usepackage{color}
\usepackage{faktor}
\usepackage{hyperref}
\usepackage{tikz-cd}
\usepackage{mathtools}
\usepackage{enumitem}

\newtheorem{thm}{Theorem}
\newtheorem{lemma}[thm]{Lemma}
\newtheorem{prop}[thm]{Proposition}
\newtheorem{cor}[thm]{Corollary}
\newtheorem{defi}[thm]{Definition}
\newtheorem{conj}[thm]{Conjecture}
\newtheorem{ques}[thm]{Question}

\theoremstyle{remark}
\newtheorem{ex}[thm]{Example}
\newtheorem{rmk}[thm]{Remark}

\newcommand{\Z}{\mathbb{Z}}
\newcommand{\Q}{\mathbb{Q}}

\newcommand{\F}{\mathbb{F}}
\newcommand{\D}{\mathbb{D}}
\newcommand{\s}{\mathfrak{s}}
\newcommand{\XX}{\mathbb{X}}
\newcommand{\OO}{\mathbb{O}}

\newcommand{\spinc}{\mbox{Spin}^c}
\newcommand{\zetap}{\Z_p}
\newcommand{\lp}{L(p,q)}
\newcommand{\conc}{\mathcal{C}^Y}
\newcommand{\concom}[1]{\mathcal{C}_{#1}^Y}
\newcommand{\nodi}{\mathcal{K}(Y)}
\newcommand{\concsf}{\mathcal{C}}
\newcommand{\qconc}{\widetilde{\mathcal{C}}^Y}
\newcommand{\sage}{\raisebox{-0.7ex}{\includegraphics[width=1cm]{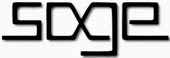}}}

\title{On concordances in 3-manifolds}
\author{Daniele Celoria}
\date{}
\address{Department of Mathematics, Oxford University, Oxford, U.K.}
\email{Daniele.Celoria@maths.ox.ac.uk}

\begin{document}

\begin{abstract}
We describe an action of the concordance group of knots in $S^3$ on concordances of knots in arbitrary $3$-manifolds. As an application we define the notion of \emph{almost-concordance} between knots. After some basic results, we prove the existence of non-trivial almost-concordance classes in all non-abelian $3$-manifolds. Afterwards, we focus the attention on the case of lens spaces, and use a modified version of the Ozsv{\'a}th-Szab{\'o}-Rasmussen's $\tau$-invariant to obstruct almost-concordances and prove that each $L(p,1)$ admits infinitely many nullhomologous non almost-concor-\\dant knots.
Finally we prove an inequality involving the cobordism $PL$-genus of a knot and its $\tau$-invariants, in the spirit of \cite{sarkar2010grid}.
\end{abstract}
\maketitle

\section*{Introduction}\label{sec:intro}

A classical and extensively studied feature of knots in the $3$-sphere is the group structure induced by connected sum on concordance classes. Much is known on the concordance group $\mathcal{C}$, and many recent progresses have been made by Heegaard Floer theoretic techniques. On the other hand, concordances in manifolds other than $S^3$ lack a clear algebraic structure.

The purpose of this paper is to describe an action of the concordance group of knots in $S^3$ on concordances of knots in an arbitrary $3$-manifold. The action consists simply in taking the connected sum of a concordance class of a knot in a $3$-manifold with a concordance class of a knot in $S^3$. After showing that the action is well defined, we introduce the related notion of \emph{almost-concordance}. This can be thought of as concordance up to connect sum with knots in $S^3$.

Similar constructions and definitions have been previously considered by Rolfsen in \cite{rolfsen85} (see also Hillman's book \cite[Sec. 1.5]{hillman2012algebraic}).

Denote the set of almost-concordances in a closed, oriented $3$-manifold $Y$ by $\qconc$. We deduce, as a result of explicit computations (Sections \ref{sec:concrete} and \ref{sec:extension}), the following result establishing the non-triviality of the equivalence relation provided by almost-concordance:
\begin{thm}\label{teo:qconcnonbanalelens}
Each lens space $L(p,1)$ contains infinitely many almost-concordance classes of knots, so $| \widetilde{\mathcal{C}}^{L(p,1)}| = \infty$. Moreover, all these classes are represented by nullhomologous knots.
\end{thm}

Instead, in the case of $3$-manifolds with non-abelian fundamental group we obtain:
\begin{thm}\label{teo:qconcnonbanale}
All $3$-manifolds with non-abelian fundamental group have non-trivial almost-concordance classes.
\end{thm}

Theorem \ref{teo:qconcnonbanalelens} is established by defining the \emph{$\tau$-shifted invariant} $\tau_{sh}$ for knots in lens spaces, which is derived by the usual Ozsv{\'a}th-Szab{\'o}-Rasmussen's $\tau$-invariant.

We then make use of Hedden's generalisation \cite{hedden2008ozsvath} of the $\tau$ invariants to extend the definition of $\tau_{sh}$ to a much larger set of $3$-manifolds (Definition \ref{taushiftgener}).
We prove that $\tau_{sh}$ is unchanged under the previously defined action:
\begin{prop}\label{proptaushinv}
The $\tau$-shifted invariant $\tau_{sh}$ is an almost-concordance invariant of knots.
\end{prop}

The proof of Theorem \ref{teo:qconcnonbanalelens} then follows from the computation of $\tau_{sh}$ for a knot $\widetilde{K} \subset L(3,1)$ (Example \ref{esempio:0134}), coupled with a sequence of chain homotopies described in Section \ref{sec:extension}.

Theorems \ref{teo:qconcnonbanalelens} and \ref{teo:qconcnonbanale} can be equivalently stated in terms of non-transitivity of the concordance action; the same result in the case of the 3-torus is a direct consequence of D. Miller's paper \cite{miller1995extension} on Milnor's invariants.\\

Following a remark of A.Levine, we relate almost-concordance to $PL$-concordance (Definition \ref{def:plcobo}), define some generalisations of the slice genus, and establish the following inequality:
\begin{thm}\label{thm:plcobotau}
Suppose $K_0$ and $K_1$ are two knots in a lens space $L(p,q)$, connected by a PL-cobordism $\Sigma$. Then
$$D(\tau(K_0), \tau(K_1) )  \le \widetilde{g}_{PL}(\Sigma),$$
where $D$ is the distance on $\Z^p$ described in Definition \ref{def:distanzalattice}, and $\widetilde{g}_{PL}$ is the (cobordism) $PL$-genus of a knot (Definition \ref{defigeneri}).
\end{thm}

As an aside, we generalize a result of Kirby and Lickorish \cite{kirby1979prime}, showing in Theorem \ref{thm:conctogenuine} that every knot in a $3$-manifold is concordant to a \emph{l-prime} knot (Definition \ref{def:genuino}).

Lastly we present some results on local knots, and outline possible improvements and future directions with some conjectures.\\

\textbf{Acknowledgments:} The author wishes to thank Adam Levine for helpful comments on the first version of this paper, Marco Golla for his invaluable expertise, Paolo Lisca and Paolo Aceto for useful observations and corrections, and Agnese Barbensi for her constant support. Also, the results of Section \ref{sec:fund} would not have been possible without the helpful remarks of Eylem Zeliha Yildiz and Patrick Orson. A special thanks also to Mark Powell and Kent Orr for  suggesting and improving some references, making me learn new topics in the process.	

This project has received funding from the European Research Council (ERC) under the European Union's Horizon 2020 research and innovation programme (grant agreement No 674978).

\section{Definitions}\label{sec:defi}

In the following $Y$ is always going to denote a closed, connected  and oriented $3$-manifold. None of these conditions is actually strictly needed, but will encompass all relevant examples that will follow.\\
A knot in $Y$ is the ambient isotopy class of a smooth embedding $\iota : S^1 \hookrightarrow Y$, and  $\mathcal{K}(Y)$ will denote the set of oriented knots in $Y$. To avoid confusion we will often use the pair $(Y,K)$ to denote $K \in \nodi$.

In every such $Y$ there is a unique knot bounding an embedded disk, the unknot, denoted by $\bigcirc$. Given a knot $(Y,K)$, call $[K] \in H_1 (Y; \Z)$ the homology class it represents; if $[K] = 0$, we say it is \emph{nullhomologous}, and \emph{rationally nullhomologous} if it represents a torsion element of $H_1 (Y; \Z)$.

A nullhomologous knot is the boundary of embedded surfaces in $Y$, while in all other cases it is not\footnote{However there are several ways to define Seifert surfaces for rationally nullhomologous knots, see \cite{calegari2009knots} and \cite{rasmussen2007lens}.}.
Given two oriented knots $(Y_i,K_i)$ for $i=0,1$, we can consider their \emph{connected sum} $(Y_0 \# Y_1, K_0 \# K_1)$, given by removing  from each $Y_i$ a $3$-disk $D_i$ intersecting $K_i$ in an unknotted arc, and glueing with an orientation reversing diffeomorphism $Y_0 \setminus D_0$ to $Y_1 \setminus D_1$, matching the orientations of the knots.

A knot $(Y,K)$ is said to be \emph{local} if there exists an embedded 2-sphere in $Y$ bounding a $3$-ball $B$, such that $K \subset B$. Clearly a local knot is nullhomologous, but the converse is generally far from true (cf. example \ref{esempio:0134}). Alternatively one could define local knots as those which admit a decomposition of the form $(Y,\bigcirc) \# (S^3, K)$.

Besides isotopy, there is another weaker equivalence relation we can consider on the set of embeddings $S^1 \hookrightarrow Y$:
\begin{defi}
We say that two knots $K_0, K_1 \subset Y$ are \emph{concordant} if there exists a smooth properly embedded annulus $$A \cong S^1 \times [0,1] \hookrightarrow Y \times [0,1],$$ such that $\partial A \cap \left(Y \times [0,1] \right) =  K_0 \sqcup \overline{K}_1$. If $K_0$ is concordant to $K_1$ we write $K_0 \sim K_1$.
Concordance is an equivalence relation on $\nodi$; the set of equivalence classes is denoted by $\conc$.
\end{defi}
Concordances preserve homology classes. In other words if $K_0 \sim K_1$, then $[K_0] = [K_1]$. This implies that $\conc$ splits:
\begin{equation}\label{eq:splitting}
\conc = \bigoplus_{m \in H_1 (Y;\Z)} \concom{m} .
\end{equation}

If $Y = S^3$, we can endow $\concsf \coloneqq \mathcal{C}^{S^3} $ with a group structure. The operation is provided  by the oriented connected sum, and the inverse of a class $[K]$ is represented by the reverse mirror of $K$.

The structure of this usual concordance group, despite its importance in low-dimensional topology, still remains elusive. Recently many improvements in the understanding of $\mathcal{C}$ were made by several authors, especially by means of Heegaard Floer theoretic constructions.

For an excellent survey on concordances in $S^3$ see \cite{concordanceliv}, and for a recent survey on the interactions between Heegaard Floer homology and $\mathcal{C}$ see \cite{2015arXiv151200383H}.\\

It is clear from the definition that the connected sum of knots does not preserve the ambient manifold, so it does not provide a binary operation on $\conc$ whenever $Y \neq S^3$. Hence $\conc$ has no natural group operation, and is only a set. So there seems to be a total  loss of algebraic structure when dealing with a manifold other than $S^3$.

There is however a natural action  $\mathcal{K}(S^3) \curvearrowright \conc$ which respects the splitting into homology components. \\This action is simply defined as:
\begin{equation}\label{azione}
(S^3,K) \cdot [(Y, K^\prime)] = [(Y, K \# K^\prime)]
\end{equation}

\begin{rmk}
This action is well defined. That is, if $K_0,K_1 \in \nodi $ and $K_0 \sim K_1$, then for each knot  $(S^3,K)$ we have $K_0 \# K \sim K_1 \# K$. To see why this is the case, denote by $A$ a concordance between $K_0$ and $K_1$. Choose a simple properly embedded arc $a$ on $A$, such that the endpoints are on the two knots, and $a$ does not intersect any critical point\footnote{For the restriction to $A$ of the Morse function induced by the projection $S^3 \times [0,1] \rightarrow [0,1]$.} of $A$ (Figure \ref{fig:cactus}).
\begin{figure}
\includegraphics[width=6cm]{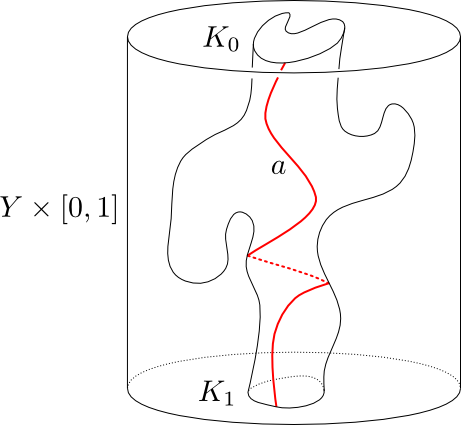}
\caption{The path $a$ on a concordance in $Y\times [0,1]$ between $K_0$ and $K_1$.}
\label{fig:cactus}
\end{figure}

Consider the product $K \times [0,1] \subset S^3 \times [0,1]$, and remove from it $D \times [0,1]$, where $D$ is a $3$-disk intersecting $K$ in an unknotted arc. Now remove a small tubular neighborhood $\nu (a)$ of $a$, and replace it with $\left( S^3 \setminus D \right) \times [0,1]$, making the boundaries of $\left( K \setminus (K \cap D) \right) \times [0,1]$ and $\left( A \setminus (\nu(a) \cap A) \right)$ coincide. The result is a concordance from $K \# K_0$ to $K \# K_1$.
\end{rmk}

Moreover this action factors through the concordance group $\mathcal{C}$:
\begin{prop}\label{prop:buonadef}
If $K_0, K_1 \in \mathcal{K}(S^3)$ and $K_0 \sim K_1$, then for each knot $(Y,K)$:
$$(S^3,K_0) \cdot [(Y, K)] \sim (S^3,K_1) \cdot [(Y, K)]$$
\end{prop}
\begin{proof}
Denote by $A \subset S^3 \times [0,1]$ an annulus realizing the concordance between $K_0$ and $K_1$. Then, as in \cite[Thm 3.3.2]{concordanceliv} we can suppose that, up to isotopy, $A$ is the product $a \times [0,1]$ for a small arc $a \subset K_0$. Remove from $S^3 \times [0,1]$ the product $D \times [0,1]$, where $D$ is a $3$-disk intersecting $K_0$ only in $a$; the complement is diffeomorphic to $\D^3 \times [0,1]$.

Take the trivial concordance $K \times [0,1] \subset Y \times [0,1]$ and remove a product $D^\prime \times [0,1]$, where $D^\prime$ is a $3$-disk intersecting $K$ in an unknotted arc. Then we just need to glue $S^3 \setminus D \times [0,1]$ to $Y \setminus D^\prime \times [0,1]$, in such a way  that the two concordances are glued along their vertical  boundaries\footnote{By vertical boundary we mean the part created by removing the intersection with the disks $D$ or $D^\prime$ times $[0,1]$.}, making the edges of the annuli coincide: the resulting annulus is a concordance from $K_0 \# K$ to $K_1 \# K$ in $Y \times [0,1]$.
\end{proof}

So we have in fact an action $\mathcal{C} \curvearrowright \conc$, which is easily seen to preserve the splitting of Equation \eqref{eq:splitting}.\\

We can introduce yet another equivalence relation on $\nodi$, by taking this $\mathcal{C}$-action into account:
\begin{defi}\label{def:almconc}
Two knots $K_0$ and $K_1$ in $Y$ are \emph{almost-concordant}, written $K_0 \dot{\sim} K_1$, if there exist two knots $K_0^\prime , K_1^\prime \subset S^3$ such that
$$K_0 \# K_0^\prime \sim K_1 \# K_1^\prime .$$
Almost-concordance is an equivalence relation on $\nodi$, and we denote by$\;\qconc$ the quotient.
\end{defi}

Clearly two concordant knots are also almost-concordant (just choose $K_0^\prime = K_1^\prime = \bigcirc$), but the converse does not hold. In particular this means that almost-concordance classes are unions of concordance classes. It is thus natural to ask whether this relation is trivial, \emph{i.e.} if there is only one almost-concordance class of knots for each pair $(Y,m)$, with $m \in H_1(Y;\Z)$. As anticipated by Theorem \ref{teo:qconcnonbanale} this is not always the case.

\begin{rmk}\label{rmk:solouno}
Note that we might also only use one knot in the definition: if $K_0 \# K_0^\prime \sim K_1 \# K_1^\prime$ then $K_0 \sim K_1 \# K^\prime$ where $K^\prime = K_1^\prime \# r\overline{K_0^\prime}$.
\end{rmk}

\begin{rmk}
It is immediate from the definition that  $|\widetilde{\mathcal{C}}^{S^3}| = 1$, that is all knots in the three-sphere are almost-concordant to each other.
In the next section we are going to outline a way to obstruct the existence of almost-concordances, after introducing a new invariant $\tau_{sh}$ capable of distinguishing them.
\end{rmk}

Shortly after the first version of this paper appeared online, A. Levine proved that almost-concordance is in fact equivalent to the older notion of $PL$-concordance; the same statement was in fact proved in greater generality by Rolfsen in \cite{rolfsen85}.
\begin{defi}\label{def:plcobo}

Two knots $K_0, K_1 \in \nodi$ are $PL$-concordant, if they are connected by a properly embedded annulus in the product cobordism $Y \times [0,1]$, which is everywhere smooth except for a finite number of singular points which are cones over knots in $S^3$.

$PL$-cobordisms are defined in an analogous manner, allowing the singular surface cobounding the knots to have non-zero genus.
\end{defi}

The equivalence between almost-concordance and $PL$-concordance goes as follows: given a $PL$-concordance between $K_0$ and $K_1$, we can connect each singular point on the concordance to either knot with simple non-intersecting arcs. Carving out a small neighborhood of these arcs yields a concordance between the original knots with extra connect-summands given by the link of the singular points.

Conversely, given an almost-concordance, we can push $K_0^\prime$ and $K_1^\prime$ inside the cobordism, capping them off with their cones.

In \cite{levine2014non} Levine proves the existence of a non-surjective satellite operator, and uses it to exhibit a knot in an homology $3$-sphere $Y$ which does not bound a PL-disk in any contractible $4$-manifold cobounding $Y$.

In Section \ref{sec:concrete} we are going to produce an invariant of almost-concordance for knots in lens spaces, which by Levine's remark is also an invariant of $PL$-concordance. We are also going to show that this invariant bounds (in an appropriate setting) from below the minimal $PL$-genus of a surface cobounding two knots.

\begin{defi}\label{defigeneri}
One can define several kinds of $4$-dimensional genera for a knot $K \in \nodi$; given a $3$-manifold $Y$ and a contractible $4$-manifold $W$ such that $\partial W = Y$, one can define the \emph{smooth slice genus} $g_{*}$, the \emph{PL genus} $g_{PL}$, the \emph{topological $4$-genus} $g_{TOP}$ and the \emph{topological PL $4$-genus } $g_{TOPL}$. They are defined as follows:
\begin{equation}
g_{\circ}^W (K)= \min \{g(\Sigma) \;|\; (\Sigma ,\partial \Sigma) \xhookrightarrow{\iota} (W,Y), \partial \Sigma = K \}
\end{equation}
With $\circ \in \{*, TOP, PL, TOPL\}$, and the corresponding embedding $\iota$ is required to be smooth if $\circ = *$, topological locally flat if $\circ = TOP$, smooth\slash locally flat except for a finite number of points which are cones over knots in $S^3$ if $\circ = PL$ or $TOPL$ respectively.

One can also define related genus by taking the minimum of the previous quantities over all $4$-manifolds $W$ cobounding a fixed $Y$, with some restriction on the algebraic topology of these fillings.

In what follows however we will be more concerned with yet another, slightly different notion of $4$-dimensional genus for knots; suppose the homology class of $K \in \nodi$ contains a ``standard" representative $T_{[K]}$. In the case where $[K] = 0 \in H_1(Y;\Z)$ one can always choose the unknot $\bigcirc$, but this is not the only possibility: in the case of lens spaces one has the simple knots $K(p,q,k)$ (in the notation of \cite{rasmussen2007lens}), which share many interesting properties with the unknot (cf. \cite{rasmussen2007lens}, and \cite{hedden2011floer}), despite being only rationally nullhomologous.

We can then define the genera
\begin{equation}\label{generitilde}
\widetilde{g}_{\circ} (K) = \min\{g(\Sigma) \;|\; \Sigma \subset Y \times [0,1], \partial \Sigma = K \sqcup \overline{T}_{[K]} \},
\end{equation}
where again $\circ$ belongs to the set $\{*, TOP, PL, TOPL\}$, and the surface $\Sigma$ is embedded in the trivial cobordism $Y\times [0,1]$ accordingly (\emph{i.e.}$\:$in the appropriate category).

We will refer to each of the $\widetilde{g}_{\circ} (K)$ genera as the \emph{$\circ$-cobordism genus} of the knot $K$.
\end{defi}

Note that in the $3$-sphere case there is no difference between $\widetilde{g}_{\circ}$ and $g_\circ$, since removing a $4$-ball from the interior of $\D^4$ yields $S^3 \times [0,1]$.

For a fixed filling $W$ (which is omitted from the notation), there are some obvious relations between these genera:
\begin{align*}
g_{TOPL} (K) \le g_{PL} (K) \;\;\;\;&\;\;\;\; g_{TOPL} (K) \le g_{TOP} (K) \\
g_{PL} (K) \le g_*(K) \;\;\;\;&\;\;\;\; g_{TOP} (K) \le g_*(K),
\end{align*}
and $g_{TOPL} = g_{PL} = 0$ for all knots in $S^3$.

It would be interesting to determine if there is a relation between $g_{TOP}$ and $g_{PL}$.
Note that if a knot has minimal PL-genus greater than $0$, then in particular, it is not concordant in a homology cobordism to any knot in $S^3$ (cf. \cite{levine2014non}).

The following elementary inequality always holds for $K \in \nodi$ and any choice of $\circ$ as above (after picking a suitable $T_{[K]}$ where available):
\begin{equation}
g_{\circ} (K) \le \widetilde{g}_{\circ} (K)  + g_{\circ}(T_{[K]}).
\end{equation}
\newline
The notion of almost-concordance is closely related to primeness of a knot. The definition makes sense also in a $3$-manifold other than $S^3$:%&
\begin{defi}\label{def:genuino}
Call a knot $(Y, K)$ \emph{l-prime} if it is not a connected sum with a knot in $S^3$, so there is no embedded\footnote{The boundary of a ball as such is usually called a \emph{Conway sphere}.} $3$-ball $B$ intersecting $K$ non trivially. By triviality of the intersection we mean that the pair $(B, K \cap B)$ is isotopic (relatively to the boundary) to the pair $(\mathbb{D}^2 \times \mathbb{D}^1 , \{0\}\times \mathbb{D}^1 )$.
In other words $(Y,K)$ is l-prime \emph{iff} for every decomposition $(Y,K) = (Y,K_0) \# (S^3,K_1)$ we have $K_1 = \bigcirc$.
\end{defi}

It was first proven by Kirby and Lickorish in \cite{kirby1979prime} that every knot in $S^3$ is concordant to a prime knot. Using the same argument of Livingston \cite{livingston1981homology}, we can obtain a generalization of their result to arbitrary  $3$-manifolds. To the best of the author's knowledge this result (despite being an almost straightforward generalisation of the techniques of \cite{livingston1981homology}) does not appear in the literature, so we include it for the sake of completeness.

\begin{thm}\label{thm:conctogenuine}
Every knot $K \subset Y$ is concordant to an l-prime knot.
\end{thm}
\begin{proof}
Consider the knot $P\subset S^1 \times \mathbb{D}^2$ shown in Figure \ref{fig:dimgenuini}.
\begin{figure}
\includegraphics[width=9cm]{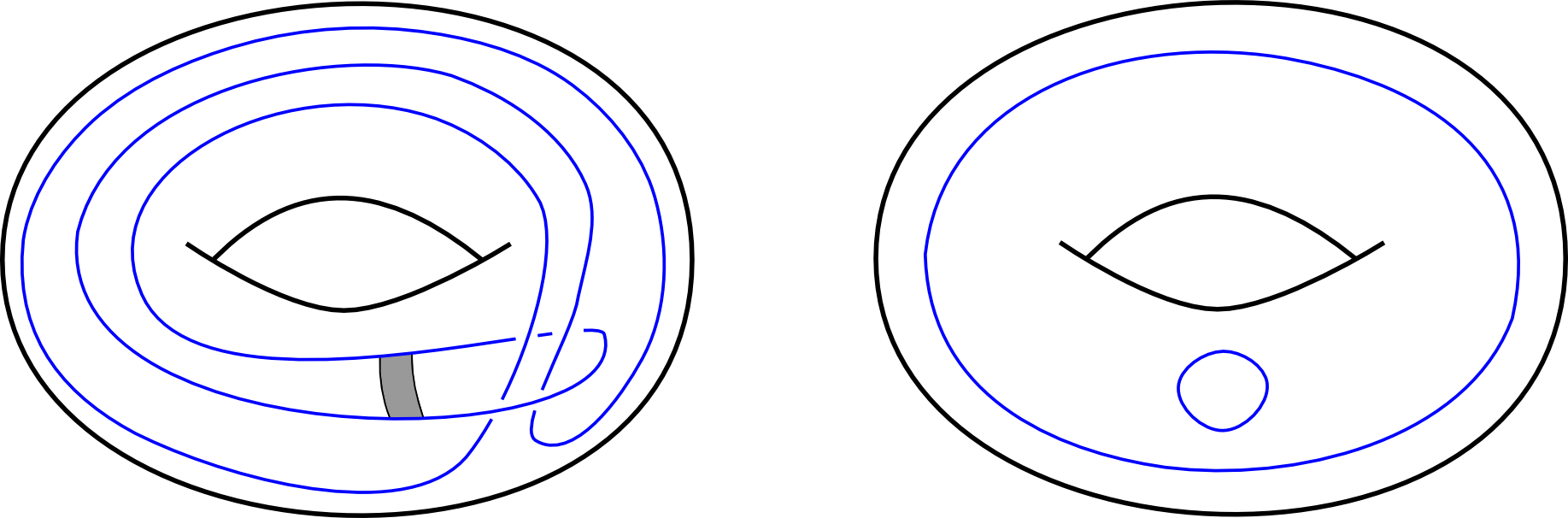}
\caption{The pattern $P$ for the satellite construction. Attaching the gray band, and capping the nullhomologous component, yields a concordance in $S^1 \times \mathbb{D}^2 \times [0,1]$ between $P$ and the core of the solid torus.}
\label{fig:dimgenuini}
\end{figure}

Given any knot $(Y, K)$, we can remove a tubular neighborhood $\nu(K)$ and glue in\footnote{We do not need to specify the framing with respect to which we are attaching the solid torus, since the result holds for any choice.} the solid torus containing the pattern, obtaining a new knot $(Y , K_P)$, the satellite of $K$ with pattern $P$.  The concordance suggested in Figure \ref{fig:dimgenuini} induces a concordance from $K_P$ to $K$. Now we just need to check that $K_P$ is in fact l-prime; %&
we start by noting that the pattern inside the solid torus is prime\footnote{This is proven in a more general setting in \cite{livingston1981homology}.}, \emph{i.e.}$\:$it cannot be split into non-trivial knots by a Conway sphere. Then we just need to argue by contradiction that any sphere giving a decomposition of $K_P$ can be isotoped away from the torus given by the boundary of the neighborhood for the original knot $K$. So necessarily the sphere would be contained in $\nu (K)$, and we can conclude by the primeness of $P$.
This can be done similarly to \cite[Thm. 4.2]{livingston1981homology}.

We sketch the construction here; call $S$ an embedded sphere giving a decomposition in two summands of $K_P$, and $R$ the annulus obtained by deleting a small neighborhood of the two points $K_P \cap S$. Generically, the intersections between $R$ and $\partial \nu (K)$ are composed by nullhomologous circles (in $R$), and circles which are parallel to $\partial R$. The first kind can be eliminated by isotopies, starting from the innermost ones.

We want to show  that there can not be any intersection which is parallel to $\partial R$; if such intersections existed, then by considering one close to $\partial R$, we would have found a disk cobounding a meridian of $\partial \nu (K)$ with intersection 1. But this is absurd, since the minimal number of intersections between $P$ and a disk cobounding a meridian $\{p\} \times S^1$ in the solid torus is 3.
\end{proof}

\section{Relation with the fundamental group}\label{sec:fund}

As independently pointed out by Eylem Zeliha Yildiz and Patrick Orson, there is a simple relation between almost-concordance and free homotopy classes of loops in a $3$-manifold $Y$:
\begin{prop}\label{prop:free}
If $K_0, K_1 \in \nodi$ are almost-concordant, then they are freely homotopic.
\end{prop}
\begin{proof}
By hypothesis there exists a smoothly embedded annulus $A$ in \\$Y \times [0,1]$ connecting $K_0 \# K^\prime_0$ to $K_1 \#K^\prime_1$, for some knots $K^\prime_0 , K^\prime_1$ in $S^3$.

Consider two regular homotopies\footnote{The $h_i$ can be \emph{e.g.} a sequence of crossing changes taking each knot to $\bigcirc$.} $h_i$ for $i=0$ and $1$, between $K^\prime_i$ and the unknot in $S^3$; we can suppose that the $h_i$ are the identity outside a ball $D$ (times an interval) containing $K^\prime_i$, and that they fix a point of $K^\prime_i$. Removing a small neighborhood of the fixed point produces a relative homotopy from $K^\prime_i$ (minus a small interval) to the unknot (again, minus a small interval) supported in $\D^3 \times [0,1]$.
One can then attach these two resulting $4$-balls to the identity cobordisms $Y\times [0,1] \setminus D^\prime_i \times [0,1]$, where $D^\prime_i$ is a $3$-ball intersecting trivially $K_i$.

The result of these attachments are two homotopies, going from $K_0 \# K^\prime_0$ to $K_0$ and from $K_1 \# K^\prime_1$ to $K_1$. Attaching these to the two ends of $A$ provides the needed homotopy from $K_0$ to $K_1$.
\end{proof}

\begin{proof}[Proof of Thm. \ref{teo:qconcnonbanale}]
It is a well known fact that free homotopy classes of loops in $Y$ are in bijection with conjugacy classes of $\pi_1 (Y)$.

If the fundamental group of a closed $3$-manifold $Y$ is not abelian, there must be at least one commutator, say $[a,b] = aba^{-1} b^{-1}$, with $a, b \in \pi_1 (Y)$ which is not the identity. Taking an embedded and connected representative for $[a,b]$ yields a nullhomologous knot which is not almost-concordant to the unknot. Therefore\footnote{See for example \cite{hempel20043} or \cite{friedlintroduction} for a classification of $3$-manifolds with non-abelian fundamental group.} ``most" $3$-manifolds must have $\left|\conc_0 \right| >1$.
\end{proof}

This implies, \emph{e.g.}$\:$if $Y$ is a $\Z HS^3$, that any smooth and simple representative of a non-trivial element of $\pi_1 (Y)$ is a nullhomologous knot which is not freely homotopic, hence not almost-concordant to the unknot in $Y$.

In the following two sections we are going to prove that a similar result holds also in a special case of closed $3$-manifolds with abelian\footnote{Recall that the only closed and orientable $3$-manifolds with non-trivial abelian fundamental group are lens spaces, the $3$-torus and $S^1\times S^2$.} fundamental groups, that is the lens spaces $L(p,1)$.

\begin{figure}[h]
\includegraphics[width=10cm]{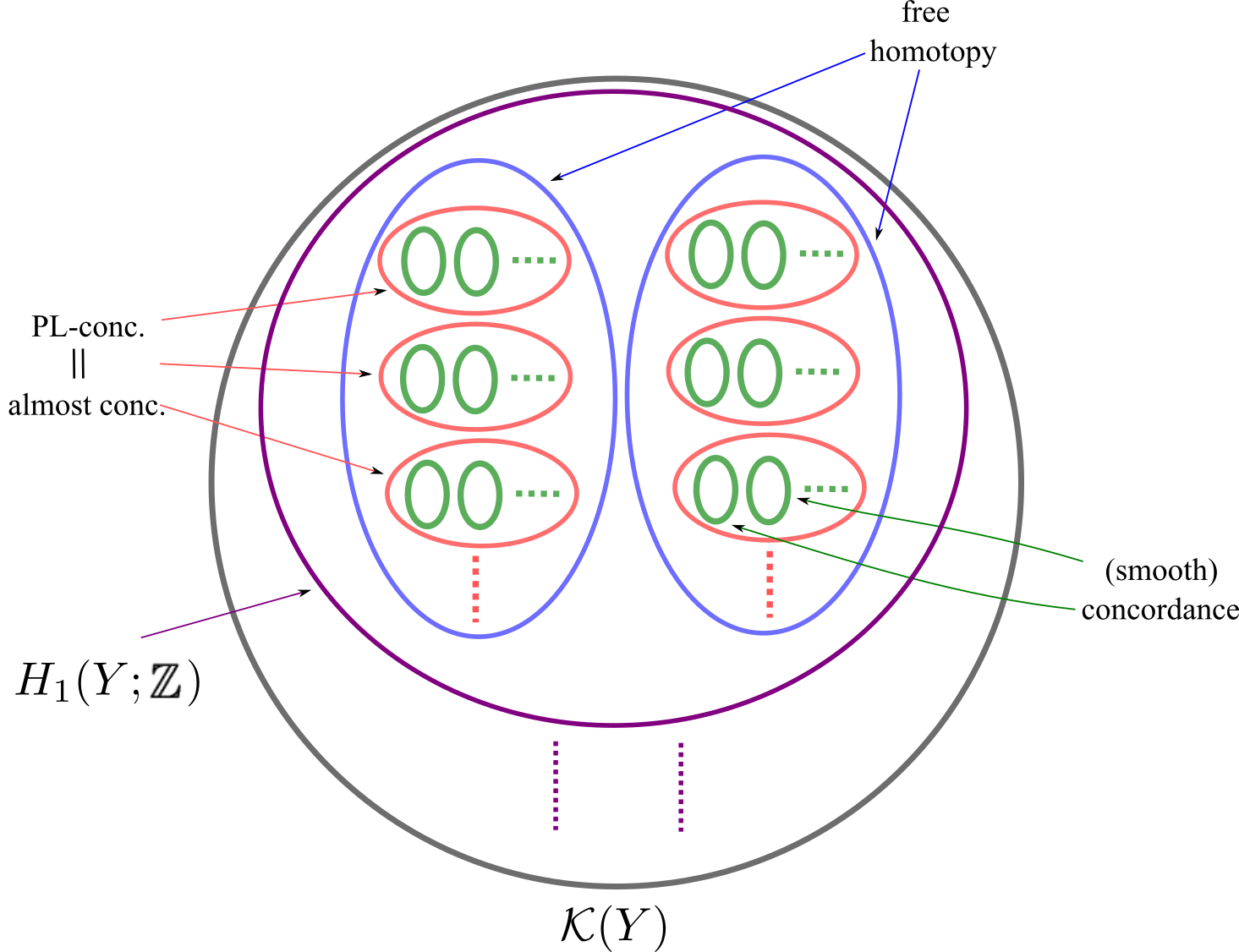}
\caption{The partition of the set of knots in a $3$-manifold into homology classes, free homotopy classes, almost-concordances and smooth concordance classes.}
\end{figure}

Proposition \ref{prop:free} implies that for each $3$-manifold $Y$ we have a splitting into free homotopy classes of loops in $Y$:
\begin{equation}
\widetilde{\mathcal{C}}^Y = \bigoplus_{l \in \faktor{\pi_1(Y)}{\mbox{\tiny{conj.}}}} \widetilde{\mathcal{C}}^Y_l
\end{equation}

\begin{rmk}
Note that the inclusions between these various classes (homology, free homotopy, PL-concordance and regular concordance) are strict, meaning that in general there are knots which are distinct in one class, but not in the larger ones.

As an example of the only not well known case, we are going to show in the next section the existence of  knots (in fact an infinite family) which are not $PL$-concordant to the unknot, but nonetheless belong to the same free homotopy\footnote{This is an easy consequence of the commutativity of the fundamental group of lens spaces.} class of $\bigcirc$. Thus the converse of Proposition \ref{prop:free} does not hold.
\end{rmk}

\section{A ``concrete" example}\label{sec:concrete}

In this section we are going to exploit Ozsv{\'a}th-Szab{\'o}-Rasmussen's Floer theoretic invariant $\tau$ to distinguish almost-concordance  classes, in the case where $Y = \lp$. We remark here that the following constructions holds in a more general context (Definition \ref{taushiftgener}), but it is generally hard to compute the invariants we will use.

On the other hand, the combinatorial versions of knot Floer homology, first developed by Manolescu, Ozsv{\'a}th and Sarkar \cite{manolescu2009combinatorial} for links in $S^3$, and by Baker, Hedden and Grigsby \cite{BGH} for links in $\lp$ are easily computable, and the latter can be used to provide many interesting examples.

We adopt the conventions of \cite{gompf4} and \cite{BGH} regarding lens spaces, according to which $\lp = S^3_{-\frac{p}{q}} (\bigcirc)$.

What follows is a brief recap on the needed results on knot Floer homology. Great sources on the subject are  \cite{holdisknot} and  \cite{manolescu2014introduction}.\\

Heegaard Floer homology is a package of invariants of $\spinc$ $3$-man-\\ifolds, introduced by {Ozsv{\'a}th and Szab{\'o} in \cite{holdisk3man}. The simplest of these invariants is denoted by $\widehat{HF} (Y,\s)$, where $\s \in \spinc (Y)$.

Soon after its definition, it was realized in \cite{holdisknot} and \cite{rasmussenknot} that a nullhomologous\footnote{The same holds for rationally nullhomologous knots, and similar results hold also in the general case, see \cite{2010arXiv1012.3088S}.} knot $(Y,K)$ induces a filtration on the complexes $\widehat{CF} (Y,\s)$ computing the Heegaard Floer homology group of $(Y,\s)$.\\
The filtered chain homotopy type of these filtered complexes, denoted by $\widehat{CFK} (Y,K,\s)$, is an invariant of the triple $(Y,K,\s)$. In particular to each such triple $(Y,K,\s)$, with $\s \in \spinc (Y)$ and $K \in \nodi$, we associate a relatively bigraded group $\widehat{HFK}(Y,K,\s)$, finitely generated over $\F = \Z /2 \Z$. This is just the homology of the graded object associated to the filtered chain homotopy type of $\widehat{CFK}(Y,K,\s)$.

The two gradings are known as the Maslov and Alexander degrees. The first one can be thought of as an homological degree (it decreases by 1 under the action of the differential), while the latter is the degree associated to the filtration induced on $\widehat{CF}(Y,\s)$ by $K$.

If $Y$ is a lens space (or more generally a rational homology $3$-sphere), the gradings lift to an absolute $\Q$-valued bigrading by results of \cite{absolutely}.

An $L$-space $\overline{Y}$ is a rational homology $3$-sphere such that for each $\s \in \spinc (\overline{Y})$:
$$\mbox{rk}_\F \left(\widehat{HF}(\overline{Y},\s) \right) = 1.$$

All lens spaces are $L$-spaces, and in the following we will fix an identification of $\spinc (\lp)$  with $\Z / p \Z$, as described in  \cite{absolutely}.\\

Knot Floer homology is  known (see \emph{e.g.} \cite{holdisknot}) to satisfy a formula\footnote{We do not specify here the various conventions involved for $\spinc$ structures, since in what follows we will only deal with the case $Y_0 = S^3$.} for the  connected  sum of two knots; if  $ (Y,K,\s) = (Y_0,K_0,\s_0)\#(Y_1,K_1,\s_1)$, then
\begin{equation}\label{eqn:connectedsum}
\widehat{HFK} \left( Y,K,\s \right) \cong \widehat{HFK} \left( Y_0 , K_0 ,\s_0 \right) \otimes \widehat{HFK} \left(Y_1, K_1 ,\s_1 \right).
\end{equation} %Equation (7) could be confusing.  I think you're referring to absolute SpinC structures in the summation; however, there is a unique one so in all cases the sum is over one element.????
The isomorphism on the complex level is a filtered chain homotopy equivalence.

The knot Floer homology of the unknot in a lens space is readily seen to be
\begin{equation}\label{eqbanale}
\widehat{HFK} (\lp, \bigcirc ,\s) \cong \F_{[d(p,q,\s),0]}
\end{equation}
where the subscript of the module indicates the bidegree (Maslov, Alexander) of the generator, and $d(p,q,\s)$ is a rational number known as the correction term\footnote{See \cite{absolutely} for the definition and a recursive formula.} for $L(p,q)$ in the $\spinc$ structure $\s$.
\begin{rmk}\label{rmk:nongenuino}
If a knot $K$ in an $L$-space $\overline{Y}$ is such that
$$\mbox{rk}_\F \left( \widehat{HFK} (\overline{Y},K,\s)\right) =1 \;\; \mbox{ and } \;\; \mbox{rk}_\F \left( \widehat{HFK} (\overline{Y},K,\s^\prime)\right) \neq 1$$ for some $\s,\s^\prime \in \spinc (Y)$, then it is l-prime by Equations \eqref{eqn:connectedsum} and \eqref{eqbanale}, coupled with the unknot detection of $\widehat{HFK}$ in $S^3$ (first proved in \cite{ozsvath2004genus}).
\end{rmk}

The main tool we are going to use in order to study the notions defined in the previous section will be a modified version of the $\tau$-invariant. This invariant was first defined for knots in the $3$-sphere in the holomorphic setting in \cite{ozsvath2003knot}, and it has proven to be extremely useful since. It is a concordance invariant\footnote{In fact it is an homomorphism $\tau :\mathcal{C} \rightarrow \Z$, see also Theorem \ref{thm:tauadd}.} of knots in $S^3$, and provides a lower bound on the slice genus. Its properties can be exploited \emph{e.g.} to give a self-contained combinatorial proof of the Milnor conjecture and to exhibit exotic $\mathbb{R}^4$s (see \cite[Ch. 8]{SOS}).

\begin{defi}\label{def:tau}
For $\s \in \spinc(\lp) $, $a \in \Q$ and $K \in \mathcal{K}(\lp)$, denote by $\mathcal{F}_{a}(\lp,K,\s)$ the elements in the complex $\widehat{CFK}(\lp,K,\s)$ with Alexander degree $\le a$.
There is a natural inclusion map:
\begin{equation*}
\iota^a : \mathcal{F}_{a}(\lp,K,\s)  \hookrightarrow \widehat{CF}(\lp,\s).
\end{equation*}
The $\tau$-invariant associated to the $\spinc$ structure $\s \in Spin^c(\lp)$ for a knot $(\lp, K)$, denoted by $\tau^\s (K)$, is the minimal $a \in \Q$ such that the induced map  in homology
\begin{equation*}
\iota^a_* : H_* \left( \mathcal{F}_{a}(\lp,K,\s) \right) \hookrightarrow \widehat{HF}(\lp,\s)
\end{equation*}
is non-trivial.
Furthermore define $$\tau (K) = \left( \tau^0 (K), \ldots , \tau^{p-1} (K) \right) \in \Q^{p}.$$
\end{defi}

\begin{rmk}\label{tauraz}
Despite being $\Q$-valued, the Alexander degrees of the elements in $\widehat{CFK} (\lp,K,\s)$ differ by integers. So, for each $\spinc$ structure $\s \in \spinc (\lp)$, these degrees are in fact a subset of $\{r_\s (K) + \Z\}$ for some $r_\s (K)\in \Q$.
In each lens space, there is a canonical choice for $r_{\s}(K)$, which depends\footnote{In each homology class of a lens space there is a unique Floer simple knot, called \emph{simple knot} (see \emph{e.g.} \cite{rasmussen2007lens}); $r_{\s}(K)$ is then just the Alexander degree of the only element of its knot Floer homology.} on the homology class of $K$, the $\spinc$ structure $\s$ and the parameters $p,q$.
Moreover, if the knot is nullhomologous, the degrees are all integers, and we can choose $r_\s = 0 \; \forall \s \in \spinc (\lp)$.

In the interest of  simplicity, in what follows we are going to regard each component of the $\tau$-invariant of knots in lens spaces as being $\Z$-valued.
\end{rmk}

The following result was proved for knots in $S^3$ in \cite{ozsvath2003knot}, building on  \cite[Theorem 7.1]{holdisknot} (which instead works for general $3$-manifolds). It was then proved in full generality by Hedden in \cite[Prop. 3.6]{hedden2008ozsvath}.
We reformulate it here as follows:
\begin{thm}[\cite{hedden2008ozsvath}]\label{thm:tauadd}
Suppose $(\lp,K) = (\lp, K_0)\# (S^3,K_1)$.  Then for each $\s \in \spinc (\lp)$:
$$\tau^\s (K) = \tau^\s (K_0) + \tau(K_1)$$
\end{thm}

In other words, the $\mathcal{C}$-action shifts the $\tau$-invariants of $(\lp,K^\prime)$ in a uniform manner in each $\spinc$ structure.\\

The next theorem is a generalization of a well known result for knots in the $3$-sphere, first proven for knots in $S^3$ by Sarkar  \cite{sarkar2010grid} in a purely combinatorial setting. In the same paper it is used to give an elementary proof of the Milnor Conjecture, first proven by Kronheimer and Mrowka \cite{kronheimer1993gauge} using gauge-theoretic techniques.
The result holds with small modifications for arbitrary $3$-manifolds (see \cite{heddenunpublished}), but in what follows we will only need a version for lens spaces.
\begin{thm}[\cite{heddenunpublished}]\label{thm:cobo}
Let $\Sigma$ be a smooth cobordism of genus $g(\Sigma)$ in $\lp \times [0,1]$ between the knots $K_0 , K_1 \in \mathcal{K} (\lp)$.\\
Then $\forall \s \in \spinc (\lp)$:
$$ |\tau^\s (K_0) - \tau^\s (K_1)| \le g(\Sigma).$$
\end{thm}

\begin{cor}
Suppose $(\lp, K_0) \sim (\lp , K_1)$. Then for all $\s \in \spinc (\lp)$ $$\tau^\s (K_0) = \tau^\s (K_1),$$ that is the $p$-tuple of $\tau$-invariants is a concordance invariant.
\end{cor}
\begin{proof}
By hypothesis there is a genus-0 surface $\Sigma$ connecting $K_0$ and $K_1$ in $\lp \times [0,1]$, so for all $\s \in \spinc (\lp)$:
$$0 \le |\tau^\s (K_0) - \tau^\s (K_1)| \le g(\Sigma) = 0.$$
\end{proof}

\begin{rmk}
By adapting the techniques of \cite{SOS} and \cite{BGH},  Theorem \ref{thm:cobo} can be proven in a combinatorial fashion also for knots in lens spaces (see \cite{PHDtesi}).
\end{rmk}

We can now turn to the study of the almost-concordance classes of knots in $\lp$.
The key fact that will allow us to distinguish them is Theorem \ref{thm:tauadd}:
\begin{defi}\label{def:taushifted}
Let $(\lp , K)$ be a knot; define the \emph{shifted $\tau$-invariant} as the $p$-tuple
$$\tau_{sh} (K) = (\tau^1(K) + n , \ldots , \tau^p(K) + n)$$
where $n \in  \Z$ is the only integer\footnote{\emph{cf.} Remark \ref{tauraz}.} such that $\displaystyle \min_{\s} \{\tau^\s (K) + n\} = 0$.
\end{defi}

We can now turn to the proof of Proposition \ref{proptaushinv}.
\begin{proof}[Proof of Prop. \ref{proptaushinv}]
Using Theorem \ref{thm:tauadd}, it is immediate to show that
the $\tau_{sh}$-invariant is unchanged under the action described in Equation \eqref{azione}, hence almost-concordant knots have the same $\tau_{sh}$-invariant.
\end{proof}

\begin{prop}\label{prop:localeshif}
If $K$ is a local knot in $\lp$, then $\tau_{sh} (K) = (0,\ldots,0)$.
\end{prop}
\begin{proof}
It follows immediately from Theorem \ref{thm:tauadd}, and the fact that the unknot has trivial $\tau_{sh}$-invariant.
\end{proof}

It is possible to generalise greatly the definition of the $\tau_{sh}$-invariant, by considering Hedden's approach to $\tau$ invariants \cite{hedden2008ozsvath}:

\begin{defi}\label{taushiftgener}
Given a knot $K \in \nodi$, one can choose an ordered tuple of elements $(x_1, \ldots, x_m) \in \left( \widehat{CF}(Y) \right)^m$ which are non-trivial and distinct in the homology $\widehat{HF}(Y)$. To each of these elements we can associate a numerical invariant $\tau_{[x_i]}(Y,K)$.
\end{defi}
\begin{rmk}
In the previous definition of $\tau_{sh}$, we chose $[x_i] = \widehat{HF}(L(p,q),\s_i)$, that is the only element of $\widehat{HF}(L(p,q))$ in the $i$-\emph{th} $\spinc$ structure.\\
The $\tau_{[x_i]}(Y,K)$ invariants have the same behaviour (\cite[Prop. 3.6]{hedden2008ozsvath}) under connected sum with knots in the three-sphere as the ``regular" $\tau$-invariants, hence can be used to obstruct the existance of almost-concordances as well.
\end{rmk}

In particular if $[x] \neq [y]$ are two non-trivial classes in $\widehat{HF}(Y)$ for a $\Q HS^3$ $Y$, and $K_0,K_1$ are two knots in $\nodi$ such that
\begin{equation}
\left(\tau_{[x]}(Y,K_0),\tau_{[y]}(Y,K_0)\right) \neq \left(\tau_{[x]}(Y,K_1) + r,\tau_{[y]}(Y,K_1)) + r\right),
\end{equation}
where $r\in \Q$ is the only rational such that $\tau_{[x]}(Y,K_0) = \tau_{[x]}(Y,K_1) + r$, then $K_0 $ is not almost-concordant to $ K_1$.

\begin{ex}\label{esempio:0134}
Adapting the grid-diagrammatic approach\footnote{See also \cite{manolescu2007combinatorial} and \cite{SOS}.} developed in \cite{BGH}, we were able to compute in \cite{PHDtesi} the $\tau$-invariants for the knot $(L(3,1),\widetilde{K})$ shown in the twisted grid diagram form\footnote{For the definitions of twisted grid diagrams for knots in lens spaces see \cite{BGH}.} in Figure \ref{nodo:0134}.

The following tables provide the sets of generators for the complexes $\widehat{GC}(L(3,1), \widetilde{K},\s)$, which are combinatorially defined and  quasi-isomorphic to the complexes $\widehat{CFK}(L(3,1), \widetilde{K},\s)$ by \cite[Thm.1.1]{BGH}.

These are finitely generated $\F[U]$-modules, where $U$ denotes a graded endomorphism decreasing the Maslov and Alexander degrees of a generator by 2 and 1 respectively. Note that the resulting homology is a finitely generated $\F$-module.

There are 6 generators (as an $\F [U]$-module) for each complex \\$\widehat{GC}(L(3,1),\widetilde{K},\s)$ computing the knot Floer homology of $(L(3,1), \widetilde{K})$ in the $\spinc$ structure $\s$. We denote them by $x_i^\s$ for $i = 0, \ldots,5$ and $\s = 0,1$, and display their filtered differentials below. We omit the computations for $\s = 2$, since in that case the complex is identical to the $\s = 1$ one.

The filtration on the complex is used to compute the $\tau$-invariants according to definition \ref{def:tau}. To obtain the homology of the associated graded $\widehat{HFK}(L(3,1),\widetilde{K},\s) \cong  \widehat{GH}(L(3,1),\widetilde{K},\s)$, just delete all the differentials which do not preserve the Alexander degree (these are the grey dashed lines in Figure \ref{fig:complessi}).

\begin{tabular}{ccc}\label{contazzi}
Generator & $\left(M, A \right)$ &  Differential\\
\hline
$\spinc$  degree = 0 & & \\
\hline
& &\\
$x_0^0$ & $\left(\frac{3}{2}, 1 \right)$   & $\partial (x_0^0) = x_1^0 + x_2^0$\\
$x_1^0$ & $\left(\frac{1}{2}, 0 \right)$   & $\partial (x_1^0) = Ux_0^0 + x_3^0 + x_4^0$\\
$x_2^0$ & $\left(\frac{1}{2}, 0 \right)$   & $\partial (x_2^0) =  Ux_0^0 + x_3^0 + x_4^0$\\
$x_3^0$ & $\left(-\frac{1}{2}, -1 \right)$ & $ \partial (x_3^0) = U (x_1^0 + x_2^0)$\\
$x_4^0$ & $\left(-\frac{1}{2}, -1 \right)$ & $\partial (x_4^0) = 0$\\
$x_5^0$ & $\left(-\frac{3}{2}, -2 \right)$ & $\partial (x_5^0) = U x_4^0$ \\
& &\\
\hline
$\spinc$  degree = 1 & & \\
\hline
& &\\
$x_0^1$ & $\left(\frac{7}{6}, 0 \right)$   & $ \partial (x_1^1) = x_1^1 + x_2^1$\\
$x_1^1$ & $\left(\frac{1}{6}, 0 \right)$   & $\partial (x_1^1) = x_4^1 + x_5^1 $\\
$x_2^1$ & $\left(\frac{1}{6}, 0 \right)$   & $\partial (x_2^1) = x_4^1 + x_5^1$\\
$x_3^1$ & $\left(\frac{1}{6}, -1 \right)$  & $\partial (x_3^1) = x_4^1 + x_5^1$\\
$x_4^1$ & $\left(-\frac{5}{6},  -1\right)$ & $\partial (x_4^1) = U (x_1^1 + x_3^1)$\\
$x_5^1$ & $\left(-\frac{5}{6}, -1 \right)$ &  $\partial (x_5^1) = U (x_1^1 + x_3^1) $   \\
& &\\
\end{tabular}\\

\begin{center}
\begin{figure}
\includegraphics[width = 9cm]{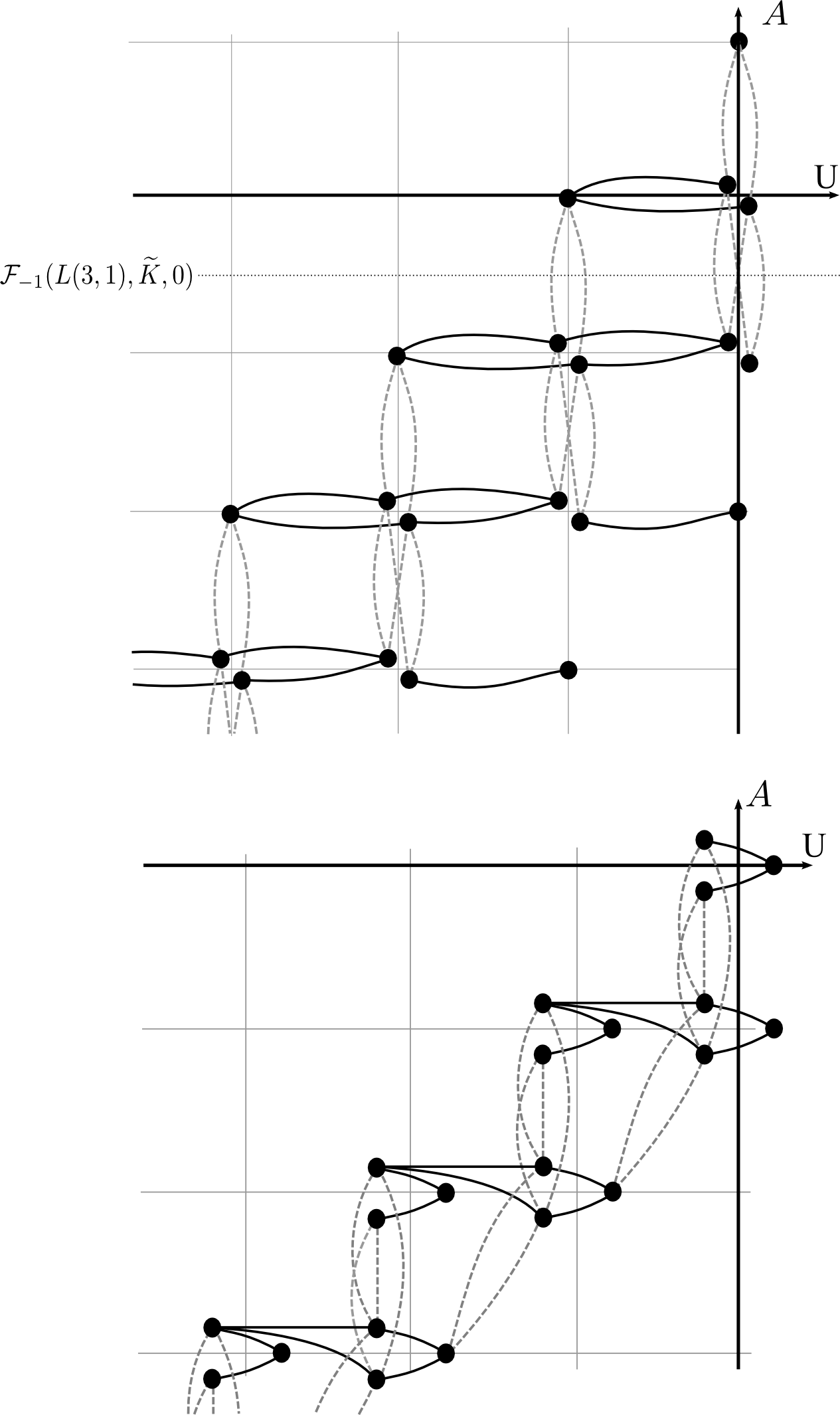}
\caption{The complexes $\widehat{GC}(L(3,1),\widetilde{K},\s)$ for $\s = 0$ (top) and $\s = 1$ (bottom). The axes are labeled by Alexander degree and powers of $U$. The dotted line highlights the filtration level where the map\newline \hspace{\linewidth} $\;\;\;\;\;\;\; \iota^a_* : H_* \left( \mathcal{F}_{a}(L(3,1),\widetilde{K},0) \right) \hookrightarrow \widehat{HF}(L(3,1),0)\;\;\;\;\;\;\;\;\;\;\; $ becomes surjective.}
\label{fig:complessi}
\end{figure}
\end{center}

We display the homology of the associated graded complex here:
\begin{equation}
\widehat{HFK} (L(3,1),\widetilde{K},\s) =
\begin{cases}
\F_{\left[ \frac{3}{2} , 1 \right]} \oplus\F_{\left[ \frac{1}{2} , 0\right]} \oplus \F_{\left[ -\frac{1}{2} , -1\right]}  & \mbox{ if } \s = 0\\
\F_{\left[ \frac{1}{6},0\right]}  & \mbox{ if } \s = 1\\
\F_{\left[ \frac{1}{6},0\right]}  & \mbox{ if } \s = 2\\
\end{cases}
\end{equation}
$\widetilde{K}$ is a nullhomologous knot, and  $$\tau (\widetilde{K}) = (-1,0,0) \;\;\;\; \Longrightarrow \;\;\;\; \tau_{sh}(\widetilde{K}) = (0,1,1).$$

In particular this means that $\widetilde{K}$ is not even almost-concordant to a local knot. Furthermore, since it has non-trivial knot Floer homology only
in the $\spinc$ structure 0, by Remark \ref{rmk:nongenuino} it can not be the connected sum with a knot in $S^3$, hence it is also l-prime.

\begin{rmk}
The computation of the $\tau$-invariants displayed in Figure \ref{fig:complessi} was partially aided by a grid homology calculator for knots in lens spaces. This program was developed in \sage, and a GUI version is freely  available on my homepage \url{http://poisson.phc.dm.unipi.it/~celoria/#programs}.
\end{rmk}

\begin{figure}[h]
\includegraphics[width=6cm]{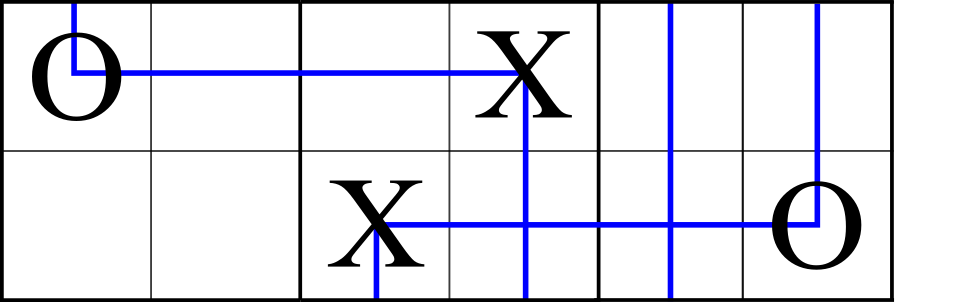}
\caption{A grid diagrammatic representation of the knot $\widetilde{K}$ in $L(3,1)$.}
\label{nodo:0134}
\end{figure}
\end{ex}

\section{Extension to $L(p,1)$}\label{sec:extension}

Now we are going to extend the result obtained from the previous computations to obtain a nullhomologous knot $\widetilde{K}_p \subset L(p,1) $ for each $p\ge 3$, such that $\widetilde{K}_p  $ is not almost-concordant to $\bigcirc$.

The knots $\widetilde{K}_p$ are constructed by \emph{expanding} (Definition \ref{def:expansion}) the grid diagram presentation for the knot $\widetilde{K} = \widetilde{K}_3$ of example \ref{esempio:0134}. Every such knot can be described by a five-tuple of integers $(p,X_0,X_1,O_0,O_1)$; \\$p$ is the parameter of $L(p,1)$, while the other four values describe the position of the $\mathbb{X}$ and $\mathbb{O}$ markings appearing in the grid.
In our case all the knots $K_p$ can be described by $\mathbb{X} = (2,3)$ and $\mathbb{O} = (5,0)$, so one can easily see that $K_3 = \widetilde{K}$ from example \ref{esempio:0134}.

\begin{rmk}\label{rmklevine}
As pointed out by Levine \cite{levineprivate}, a knot in a lens space whose homology in each $\spinc$ structure is isomorphic to either the trefoil or the unknot (and both cases occur) has non-trivial $\tau_{sh}$-invariant. This is due to the fact that if $\widehat{HFK}(\lp,K, \s)\cong \widehat{HFK}(S^3,3_1)$ (up to a shift in the Alexander grading), the spectral sequence (see \cite[Lemma 3.6]{holdisknot}) $$\widehat{HFK}(\lp,K, \s) \Longrightarrow \widehat{HF}(\lp,\s) \cong \F$$ has two consecutive terms canceling each other out, leaving a surviving generator in Alexander degree $\pm 1$.
\end{rmk}

\begin{defi}\label{def:expansion}
The operation of \emph{expansion} consists in adding an $n\times n$ box on the right of an $n$-dimensional grid representing a knot in $L(p,1)$, taking it to a new grid, with the same dimension, but representing a knot in $L(p+1,1)$ (see Figure \ref{nodo:kappap}).
\end{defi}

In what follows we are going to exhibit a (filtered) chain homotopy between the grid complex for $\widetilde{K}_p$ to $\widetilde{K}_{p+1}$ in the $\spinc$ structures $0$ and $1$. Using Remark \ref{rmklevine}, we are going to prove that each of these knots have non-trivial $\tau_{sh}$-invariants, hence they can not be almost-concordant to the unknot.

\begin{figure}[h]
\includegraphics[width=9cm]{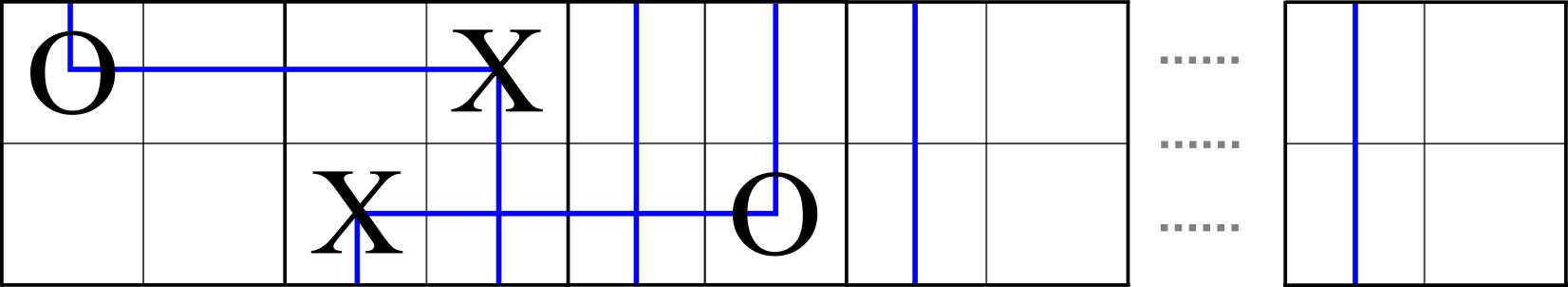}
\caption{A grid diagram representing the knots $K_p \in \mathcal{K}(L(p,1))$ (the grid is composed by $4p$ squares). }
\label{nodo:kappap}
\end{figure}

\begin{rmk}
This expansion procedure can be carried out in slightly greater generality, and is going to be fully detailed in an upcoming paper.
\end{rmk}

We find it more convenient to prove the following statements in terms of the \emph{tilde} version\footnote{Cf. \cite[Ch. 4]{SOS}.} of grid complex and homology, denoted by $\widetilde{GC}(L(p,1),K, \s)$ and $\widetilde{GH}(L(p,1),K, \s)$ respectively. \\In the case at hand (so for grid diagrams of dimension 2), we have an isomorphism
\begin{equation}
\widehat{HFK}(L(p,1),K,\s) \cong \widetilde{GH}(L(p,1),K, \s) \otimes \left(\F_{(0,0)} \oplus \F_{(-1,-1)} \right).
\end{equation}

The differential $\widetilde{\partial}$ of the tilde version of grid homology can be represented by empty\footnote{That is not containing any $\XX$ or $\OO$ marking or component of the generators.} embedded rectangles in the grid.

The computation of the differentials is a easy but quite tedious exercise, we report here the results. First denote by $G_p$ the grid diagram describing the knot $K_p$;
each generator of the complex $\widetilde{GC}(G_p)$ can be regarded as a pair $(\sigma,(a,b)) \in \mathfrak{S}_2 \times \{0,\ldots, p-1\}^2$.

In order to ease a bit the notation, we call $x_{a,b} = (Id,(a,b))$ and $y_{a,b} = ((12)),(a,b))$, and refer to $(a,b)$ as the \emph{$p$-coordinates} of the generator (note that in the following they will be regarded as integers $\text{mod}\;p$).

The assumption on the $\spinc$ structure being 0 implies (\cite[Sec. 2.2]{BGH}) that $ a+b \equiv 2 \; (mod \;p)$. We can thus divide the generators according to whether $0 \le a,b\le 2$, or $a,b \ge 3$.

If $\s = 0$ the differential is:
\begin{equation}\label{eqdiff1}
\widetilde{\partial}(x_{a,b}) =
\begin{cases}
0 & \mbox{ if } a,b \in \{0,1,2\}\\
y_{b,a} + y_{a-1,b+1} & \mbox{ if } 3 \le a\le b\\
y_{a,b} + y_{b-1,a+1} & \mbox{ if } 3 \le b \le a
\end{cases}
\end{equation}
\begin{equation}
\widetilde{\partial}(y_{a,b}) =
\begin{cases}
0 & \mbox{ if } a,b \in \{0,1,2\} \mbox{ or } a = p-1\\
x_{a,b} + x_{b,a} & \mbox{ if } 3 \le a < b\\
x_{a+1,b-1} + x_{b-1,a+1} & \mbox{ if } 3 \le b \le a < p-1
\end{cases}
\end{equation}

If instead $\s = 1$, the generators $(\sigma,(a,b)) \in \mathfrak{S}_2 \times  \{0,\ldots, p-1\}^2$ satisfy $a + b \equiv 3 \; (mod\; p)$.

If $p\ge 5$, one can divide the generators in two sets as before: those in which both indices $a,b$ are strictly smaller than $4$, and the other ones in which both $p$-coordinates are $\ge 4$. The case in which $p=4$ can be worked out by hand (or even better with a computer), and one can prove that there are (filtered) chain homotopies relating the complexes of $\widetilde{K}_3, \widetilde{K}_4$ and $\widetilde{K}_5$ for $\s=1$. \\The corresponding differentials for $p\ge 5$ are:
\begin{equation}
\widetilde{\partial}(x_{a,b}) =
\begin{cases}
0 & \mbox{ if } (a,b) = (0,3),(1,2),(2,1)\\
y_{2,1} + y_{0,3} & \mbox{ if } (a,b) = (3,0)\\
y_{b,a} + y_{a-1,b+1} & \mbox{ if } 4\le a\le b \le p-1\\
y_{a,b} + y_{b-1,a+1} & \mbox{ if } 4 \le b\le a\le p-1
\end{cases}
\end{equation}
\begin{equation}\label{eqdiff2}
\widetilde{\partial}(y_{a,b}) =
\begin{cases}
0 & \mbox{ if } (a,b) = (0,3),(2,1) \\
x_{0,3} & \mbox{ if } (a,b) = (3,0),(p-1,4)\\
x_{2,1} + x_{1,2} & \mbox{ if } (a,b) = (1,2)\\
x_{a,b} + x_{b,a} & \mbox{ if } 4 \le a < b\\
x_{a+1,b-1} + x_{b-1,a+1} & \mbox{ if } 4 < b \le a
\end{cases}
\end{equation}

In both cases we can apply the \emph{cancellation lemma} (see \emph{e.g.} \cite{krcatovich2015reduced}) to ``cancel" the two rightmost generators (see Figure \ref{evolutioncpx1}).\\
We obtain two complexes which are chain homotopic to both $\widetilde{GC}(G_p,\s)$ and $\widetilde{GC}(G_{p-1},\s)$.

\begin{figure}[h]
\includegraphics[width=11cm]{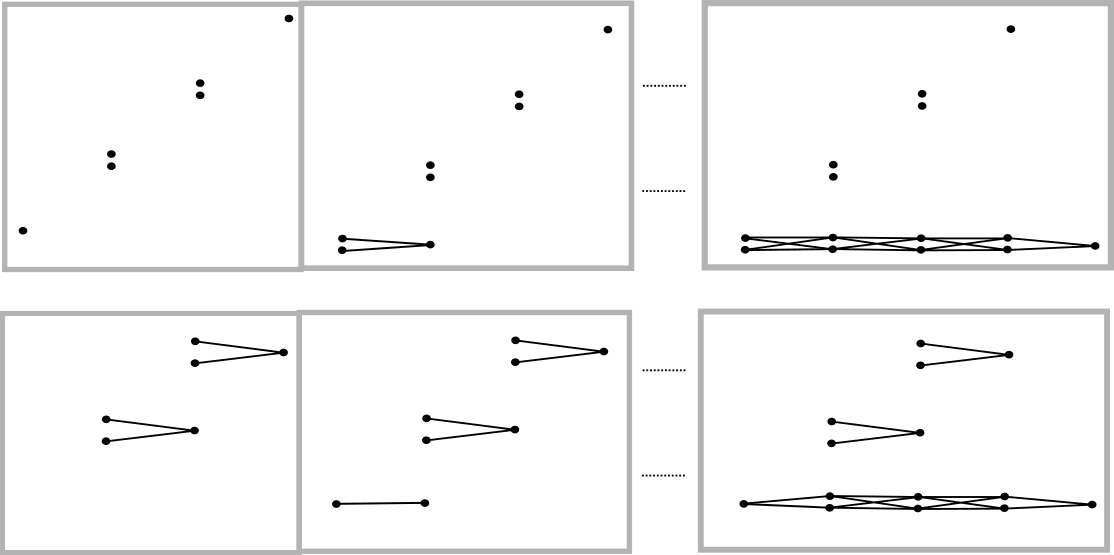}
\caption{The``evolution"of the complexes $\widetilde{GC}(L(p,1),K_p,\s)$ for $\s = 0$ (top part) and $\s = 1$ (on the lower part), and $p = 3,4$ and $7$ from left to right. In this case increasing the dimension of the grid produces chain homotopic complexes. }
\label{evolutioncpx}
\vspace{0.5cm}
\includegraphics[width=9cm]{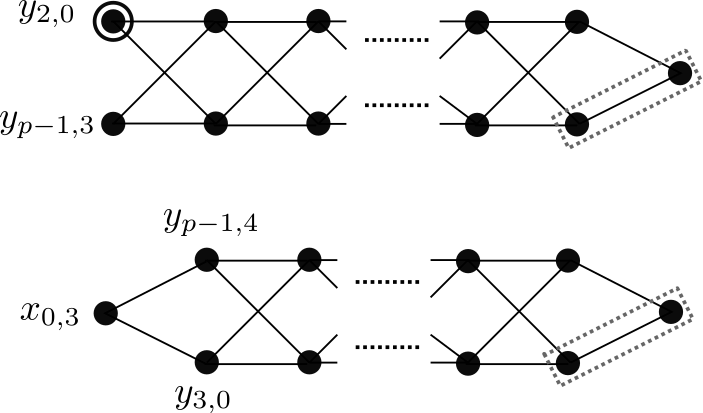}
\caption{A closer look to the lower parts in the last figure. The circled generator is the only one surviving in homology; cancelling the two generators enclosed by the dotted grey line produces a chain homotopy to the previous complex.}
\label{evolutioncpx1}
\end{figure}

Equivalently, the homology of the complexes in the two relevant $\spinc$ structures in minimal Alexander degree is always represented by the cycle $[y_{2,0}]$ if $\s = 0$ and is acyclic for $\s=1$ (see Figures \ref{evolutioncpx} and \ref{evolutioncpx1}).

In these last sections we have proved that $\widetilde{\mathcal{C}}^{L(p,1)}_0$ is non-trivial, by exhibiting the null-homologous knots $K_p \in \mathcal{K}(L(p,1))$, and showing that they are not almost-concordant to the unknot:
\begin{prop}\label{ellepiuno}
All the knots $\widetilde{K}_p \in \mathcal{K}(L(p,1))$ represent non-trivial almost-concordance classes in $\widetilde{\mathcal{C}}_0^{L(p,1)}$.
\end{prop}
\begin{proof}
The results of this section imply that each of the $\widetilde{K}_p$ (for $p\ge 3$) are nullhomologous knots in $L(p,1)$, such that 
$$\tau^0 (\widetilde{K}_p) = \tau^0 (\widetilde{K}_3) = -1 \neq 0 = \tau^0 (\widetilde{K}_3) = \tau^0 (\widetilde{K}_p).$$
In particular their $\tau_{sh}$-invariants are non-trivial.
\end{proof}

\begin{rmk}\label{rmk:l21}
We left out the case of $L(2,1)$ up to now, but it is possible to prove that there is a nullhomologous knot $K_2 \in \mathcal{K}(L(2,1))$ such that $\widehat{GH}(L(2,1),K_2,0) \equiv \widehat{GH}(3_1)$ and $\widehat{GH}(L(2,1),K_2,1) \equiv \widehat{GH}(\bigcirc)$. The dimension of a minimal grid representing $K_2$ is $3$, and it can be represented as $\XX = (0,1,2), \OO = (2,3,4)$ (see Figure \ref{fig:nododue}).
\begin{figure}
\includegraphics[width=7cm]{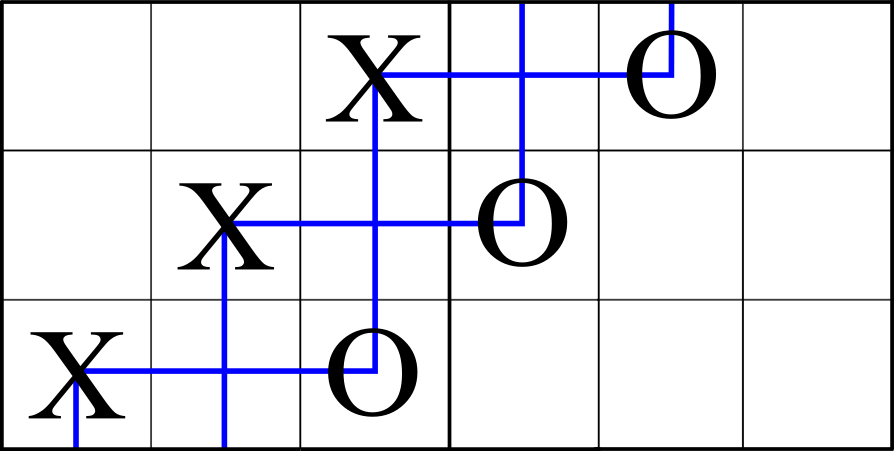}
\caption{A grid representing the ``smallest" knot in $L(2,1)$ with non-trivial $\tau_{sh}$ invariant. }
\label{fig:nododue}
\end{figure}
\end{rmk}

Now we can just recall a result by Hedden \cite{hedden2009knot} adapted to our situation to actually obtain infinitely many non-trivial almost-concordance classes in each $L(p,1)$.

Since we are dealing with nullhomologous knots we do not need to worry about framing issues\footnote{The ambiguity in this case is removed by choosing the Seifert longitude for the pattern attachment.}, which arise in the general case. Note also that any choice of the parameters for the cabling applied to $K_p$ yield another nullhomologous knot.

\begin{thm}[2.2 of \cite{hedden2009knot}]\label{teo:cabling}
Let $K \in \mathcal{K}(L(p,1))$, and choose a sufficiently large integer $n$; then if $K_{m, -mn+1}$ denotes the $(m, mn+1)$-cable of $K$, the following holds for each $\s \in \spinc (L(p,1))$:
\begin{equation}
\tau^\s (K_{m,-mn+1}) = 
\begin{cases}
m \tau^\s (K) + \frac{mn(m-1)}{2} + m-1 \; \mbox{ or }\\
m \tau^\s (K) + \frac{mn(m-1)}{2}.
\end{cases}
\end{equation}
\end{thm}

\begin{proof}[Proof of Thm. \ref{teo:qconcnonbanalelens}]
The chain homotopies
from Section \ref{sec:extension}, together with Remark \ref{rmk:l21} ensure the existence of a nullhomologous knot $\widetilde{K}_p$ in each $L(p,1)$ ($p\ge2$) which is not almost-concordant to the unknot. Theorem \ref{teo:cabling} ensures that by cabling the knots $\widetilde{K}_p$  we obtain other knots whose first two components of the $\tau$-invariant differ, hence represent non-trivial almost-concordance classes.
It is also easy to show that even keeping the parameter $m$ fixed, \emph{e.g.}$\; m=2$, the $\tau$-invariants one obtains by iterate cabling are all distinct.
Together with Proposition \ref{ellepiuno} this concludes the proof of the Theorem.
\end{proof}

\begin{rmk}\label{rmk:tanteclassi}
An extensive number of computer-aided computations allow one to prove that for $p \le 20$ it is also possible to exhibit non-trivial almost-concordance classes in each homology class of each lens space $L(p,q)$.
\end{rmk}

\begin{rmk}
It is also possible to use methods and computations by Grigsby (\cite{grigsby2006knot}) and Levine \cite{levine2008computing}, to establish the existence of non-trivial almost-concordance classes in rational homology spheres (many of which  are not lens spaces).
\end{rmk}

\begin{rmk}
The behavior of the $\tau$-invariants under mirroring (described in \cite[Prop. 3.5]{hedden2008ozsvath}), coupled with the previous computations imply the existence of a knot $\overline{K}$ in $L(3,2)$ such that $\tau_{sh} (\overline{K}) = \tau (\overline{K}) = (1,0,0)$.
More generally, the procedure detailed in this section allows to exhibit non-trivial almost-concordance classes in $L(p,p-1)$ as well.

The present techniques do not seem however to be directly applicable to the more general case of $q \neq \pm 1$; in particular the \emph{expansion} procedure described in Definition \ref{def:expansion} does not work whenever $q \neq \pm 1$. It seems nonetheless likely (cf. Remark \ref{rmk:tanteclassi}) that even in these other cases one can obtain infinitely many almost-concordance classes in each lens space.
\end{rmk}

\section{Final remarks}

We collect in this last section some results on local knots and prove Theorem \ref{thm:plcobotau}.

It seems natural to ask whether locality and concordance are in some way related. The answer to the following question says that this is not the case.
\begin{ques}\label{questionlocal}
Can a local knot be concordant to a non local knot?
\end{ques}
The answer is positive: the easiest way to produce infinitely many examples was suggested by Marco Golla. Take a non-local and non-nullhomologous knot $(Y,K)$, with $Y$ a rational homology $3$-sphere, and a ribbon pattern $P \subset S^1 \times \D^2$, as in Figure \ref{fig:patternlocale}.

Suppose moreover that $K$ has rational genus\footnote{For the definition see \cite{calegari2009knots}.} $g_{\Q} (K) >0$. Then consider $K_P$, the satellite of $K$ with pattern $P$, embedded in $Y \times \{0\} \subset Y \times [0,1]$. Note that $K_P$ bounds a ribbon disk, and it is nullhomologous in $Y$. Push the ribbon disk inwards, and remove a small disk from its interior. Tubing the boundary of the removed disk to $Y \times \{1\}$ provides the needed concordance from $K_P$ to $(Y ,\bigcirc)$.

Now we need to prove the non-locality of $K_P$; suppose there existed an embedded 2-sphere $S\subset Y$ bounding a ball containing $K_P$.

If the sphere does not intersect $\partial \nu(K)$, then either it is contained in $\nu(K)$ or it contains it. The first case can be easily dismissed by looking at the pattern $P$\footnote{It can be proven that there can not be any such sphere whenever the minimal (geometric) number of intersections with a disk cobounding a meridian is not $0$, see \cite{livingston1981homology}.}.
In the second case we would have found a sphere containing $\nu (K)$, which is absurd by the non-locality of $K$. Then we just need to argue, similarly to what was done in Theorem \ref{thm:conctogenuine}, that all intersections between $S$ and $\partial \nu (K)$ can be removed up to isotopy.

These intersections appear on $S$ as simple and disjoint circles, which might be nested. Consider an innermost circle; if the corresponding intersection is nullhomologous on $\partial \nu (K)$, we can find a $3$-ball bounded by the union of a disk on $S$ and one on $\partial \nu (K)$. So this kind of intersections can be eliminated by an isotopy. There are two qualitatively different kinds of intersections which are not nullhomologous on $\partial \nu (K)$: the ones which are parallel to a meridian, and those which can instead also wind along a longitude for $\partial \nu (K)$. In the former case (again, by considering an innermost circle on $S$), we would have found a disk cobounding a meridian of $\nu(K)$ and not intersecting $K_P$, which is absurd. In the latter, the rational genus hypothesis on $K$ prevents the existence of such a disk.

\begin{figure}
\includegraphics[width=5cm]{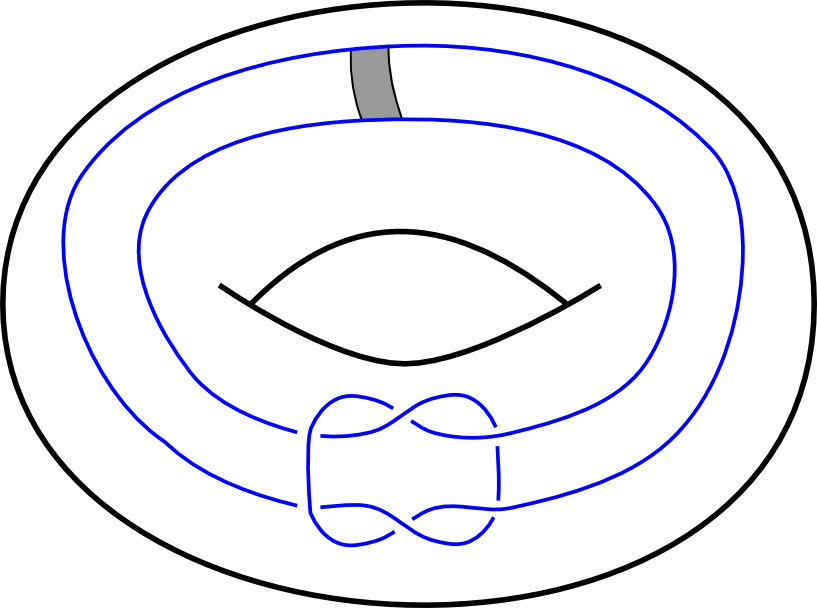}
\caption{The band attachment shown produces a concordance in $S^1 \times \D^2 \times [0,1]$ between $P$ and a pair of unknots.}
\label{fig:patternlocale}
\end{figure}

\begin{rmk}\label{obstructiontolocality}
Theorem \ref{thm:tauadd}, coupled with Equation \eqref{eqbanale} implies that the quantity
$$\max_{i,j \in \zetap} |\tau^i (K) - \tau^j (K)|$$
is an obstruction to locality (up to almost-concordance) for a knot  $(\lp,K)$, \emph{i.e.} if it is nonzero the knot can not be local.
\end{rmk}

It easy to realize that if a (almost) concordance class contains a local knot, then its cobordism $PL$-genus is equal to 0. The following easy lemma strengthens slightly the result:
\begin{lemma}
Let $K \in \nodi$ be a knot concordant to a local knot. Then for every satellite pattern $P$ we have
$$K \dot{\sim} P(K).$$
In other words, almost-concordance classes of local knots are preserved by satellites operators.
\end{lemma}
Note that we are not taking care of framings when performing the satellite construction. There is no ambiguity, since the result holds for every possible choice of framing for the knot.
\begin{proof}
Just take $K_0^\prime = \overline{P(K)}$ and $K_1^\prime = \overline{K}$ in the notation of Definition \ref{def:almconc}, to obtain the product concordance.
\end{proof}

The converse of this Lemma is quite interesting:
\begin{conj}
A knot $K \in \nodi$ is local if and only if its almost-concordance class is preserved by all satellite operators.
\end{conj}

More generally it would be interesting to find a characterization of the satellite operators which preserve the splitting into almost-concordance classes. Clearly all connected sum operators have this property, but this does not seem to be the case for other winding number 1 operators.
\newline

Before proving Theorem \ref{thm:plcobotau} we need to introduce some notation.

Consider the lattice $L(p) = \Z^{p}$ for $p \ge 2$; we want to endow $L(p)$ with a sort of path metric that encompass the behaviour of the $\tau_{sh}$-invariant under cobordism. 
\begin{defi}\label{def:distanzalattice}
Given two points $x = (x_1, \ldots , x_p)$ and $y = (y_1, \ldots, y_p) $ in $L(p)$, consider  $x^\prime = (x_2 - x_1, \ldots, x_p - x_1), y^\prime = (y_2 - y_1, \ldots, y_p - y_1)$ in $L(p-1)$; now consider the undirected graph $\mathcal{G}_p$ whose vertices are the points in $\Z^{p-1}$ and any two vertices are connected with an edge whenever their coordinates differ by $\epsilon = (\epsilon_1, \ldots, \epsilon_{p-1})$, with $\epsilon_i \in \{-1,0,1\}$ for all $i = 1, \ldots ,p-1$. In other words we are adding all ``diagonals" to the lattice. Endow $\mathcal{G}_p$ with the path metric $\widetilde{D}$, and define $$D(x,y) \coloneqq \widetilde{D} (x^\prime, y^\prime).$$
\end{defi}

\begin{proof}[Proof of Thm. \ref{thm:plcobotau}]
We want to reduce ourselves to a situation in which we can apply Theorem \ref{thm:cobo}. According to Remark \ref{rmk:solouno} genus-$g$ PL-concor-\\dance induces a regular concordance between $K_0\# K_0^\prime$ and $K_1 \# K_1^\prime$ for some $K_0^\prime, K_1^\prime \in \mathcal{K}$; the idea of the proof thus consists in estimating the minimal distance on the lattice $\Z^p$ between $\tau(K_0\# K_0^\prime)$ and $\tau(K_1 \# K_1^\prime)$ for varying $K^\prime$. This is just the minimal distance between the two equivalence classes of $\tau$-invariants induced by almost concordance.

We can consider two representatives $\widehat{K}_0$ and $\widehat{K}_1$ of the almost-concor-\\dance classes of $K_0$ and $K_1$ respectively, such that $\tau^0 (\widehat{K}_0) = \tau^0 (\widehat{K}_1) = 0$.
By Theorem \ref{thm:cobo}, each component of $\tau$ can change by at most $g$ under a cobordism of genus $g$. Hence the distance $\widetilde{D}$ on $\Z^{p-1}$ between $\widehat{K}_0$ and $\widehat{K}_1$ coincides with $$\max_{\s \in \{1,\ldots, p-1\}} \{|\tau^\s(\widehat{K}_0) - \tau^\s (\widehat{K}_1)|\},$$
which by Theorem \ref{thm:cobo} bounds from below the genus of a cobordism between the two representatives of the almost-concordance classes.
\end{proof}

In a previous version of this paper, we posed as a conjecture a question raised by A.Levine, on whether there exists a pair $(Y,m)$ such that $|\conc_m| < + \infty$. Recently Yildiz \cite{yildiz} proved, among other things, that for the pair $(S^2\times S^1, \{p\}\times S^1)$ there is only one almost-concordance class (cf. also \cite{friedl2016satellites}, \cite{nagel2017smooth} and \cite{arunima}).

\bibliographystyle{amsplain}

\providecommand{\bysame}{\leavevmode\hbox to3em{\hrulefill}\thinspace}
\providecommand{\MR}{\relax\ifhmode\unskip\space\fi MR }
\providecommand{\MRhref}[2]{
  \href{http://www.ams.org/mathscinet-getitem?mr=#1}{#2}
}
\providecommand{\href}[2]{#2}

\end{document}